\documentclass[12pt,draft,reqno]{amsart}
\usepackage{amsmath}
\usepackage{amssymb}
\usepackage{amsthm}
\usepackage{bm}
\numberwithin{equation}{section}
\newtheorem{theorem}{Theorem}[section]
\newtheorem{lemma}[theorem]{Lemma}
\newtheorem{proposition}[theorem]{Proposition}

\newtheorem{definition}[theorem]{Definition}

\theoremstyle{definition}

\allowdisplaybreaks
\begin{document}
\title[Real zeros of Hurwitz-Lerch type of Euler-Zagier double zeta functions]{Real zeros of Hurwitz-Lerch zeta and Hurwitz-Lerch type of Euler-Zagier double zeta functions}
\author[T.~Nakamura]{Takashi Nakamura}
\address[T.~Nakamura]{Department of Liberal Arts, Faculty of Science and Technology, Tokyo University of Science, 2641 Yamazaki, Noda-shi, Chiba-ken, 278-8510, Japan}
\email{nakamuratakashi@rs.tus.ac.jp}
\urladdr{https://sites.google.com/site/takashinakamurazeta/}
\maketitle

\begin{abstract}
Let $0 < a \le 1$, $s,z \in {\mathbb{C}}$ and $0 < |z|\le 1$. Then the Hurwitz-Lerch zeta function is defined by $\Phi (s,a,z) := \sum_{n=0}^\infty z^n(n+a)^{-s}$ when $\sigma :=\Re (s) >1$.
In this paper, we show that the Hurwitz zeta function $\zeta (\sigma,a) := \Phi (\sigma,a,1)$ does not vanish for all $0 <\sigma <1$ if and only if $a \ge 1/2$. Moreover, we prove that $\Phi (\sigma,a,z) \ne 0$ for all $0 <\sigma <1$ and $0 < a \le 1$ when $z \ne 1$. Real zeros of Hurwitz-Lerch type of Euler-Zagier double zeta functions are studied as well. 
\end{abstract}

%%%%%%%%%%%%%%%%%%%%%%%%%%%%%%%%%%%%%%%%%%%%%%%%%%%%
\section{Introduction and statement of main results}
%%%%%%%%%%%%%%%%%%%%%%%%%%%%%%%%%%%%%%%%%%%%%%%%%%%%

%%%%%%%%%%%%%%%%%%%%%%%%%%%%%%%%%%%%%%%%%%%%%%%%%%%%%%%%
\subsection{Reals zeros of Hurwitz-Lerch zeta functions}
%%%%%%%%%%%%%%%%%%%%%%%%%%%%%%%%%%%%%%%%%%%%%%%%%%%%%%%%
As one of a generalization of the Riemann zeta function $\zeta (s):= \sum_{n=1}^\infty n^{-s}$, the following function is well-known.
\begin{definition}[see {\cite[p.~53, (1)]{Er}}]\label{def:ler}
For $0 < a \le 1$, $s,z \in {\mathbb{C}}$ and $0< |z|\le 1$, the Hurwitz-Lerch zeta function $\Phi(s,a,z)$ is defined by
\begin{equation}
\Phi (s,a,z) := \sum_{n=0}^{\infty} \frac{z^n}{(n+a)^s}, \qquad s := \sigma + it, \quad \sigma >1 , \quad t \in {\mathbb{R}}.
\end{equation}
\end{definition}
Note that the Riemann zeta function $\zeta (s)$ and the Hurwitz zeta function $\zeta (s,a)$ are written by $\Phi (s,1,1)$ and $\Phi (s,a,1)$, respectively. The Dirichlet series of $\Phi(s,a,z)$ converges absolutely in the half-plane $\sigma >1$ and uniformly in each compact subset of this half-plane. The function $\Phi(s,a,z)$ with $z \ne 1$ is analytically continuable to the whole complex plane but $\zeta (s,a)$ is a meromorphic function with a simple pole at $s=1$. In the present paper, we show the following theorem. 
\begin{theorem}\label{th:hlz1}
We have the following:\\
${\rm{(1)}}$. Let $z=1$. Then $\Phi (\sigma,a,1) \ne 0$ for all $0 <\sigma <1$ if and only if $a \ge 1/2$. \\
${\rm{(2)}}$. Let $z \ne 1$. Then $\Phi (\sigma,a,z) \ne 0$ for all $0 <\sigma <1$ and $0 < a \le 1$. 
\end{theorem}

When $z=1$, the Hurwitz zeta function $\zeta (\sigma,a) := \Phi (\sigma,a,1)>0$ for $\sigma >1$ from the series expression $\sum_{n=0}^\infty (n+a)^{-s}$. Berndt showed that $\zeta (s,a) - a^{-s}$ has no zeros on $|s-1| \le 1$ when $0 \le a \le 1$ in \cite[Theorem 3]{Ber}. Note that $\zeta (\sigma,a) - a^{-\sigma} = \zeta (s)$ if $a=0$ and \cite[Theorem 3]{Ber} mentioned above implies that 
$$
\zeta (\sigma,a+1) = \zeta (\sigma,a) - a^{-\sigma} \ne 0
$$ 
for any $0 < \sigma <1$ and $0 \le a \le 1$. In \cite[Theorem 3]{Spira}, Spira proved that if $\sigma \le -4a-1-2[1-2a]$ and $|t|\le1$, then $\zeta (s,a) \ne 0$ except for zeros on the negative real line, one in each interval $(-2n-4a-1,-2n-4a+1)$, where $n \in {\mathbb{N}}$ and $n \ge 1-2a$. Some analogous results for the Lerch zeta function $\Phi (s,a,e^{2 \pi i \theta})$, where $0 < \theta \le 1$ are proved by Garunk\v{s}tis and Laurin\v{c}ikas in \cite{GauLa} (see also \cite[Section 8]{LauGa}). Recently, Schipani \cite{Schi} showed that $\zeta (\sigma,a)$ has no zeros and is actually negative for $0<\sigma <1$ and $1-\sigma \le a$. Note that we prove that $\zeta (\sigma,a) <0$ for any $0<\sigma <1$ when $a\ge 1/2$ during the proof process of Theorem \ref{th:hldz1} (see (\ref{ieq:hzn})). Denote the polylogarithm by ${\rm{Li}}_s (z) := z\Phi (s,1,z) = \sum_{n=1}^\infty z^n n^{-s}$. 
More than 100 years ago, Roy \cite{Roy} proved that ${\rm{Li}}_\sigma (z) \ne 0$ for all $|z| \le 1$, $z \ne 1$ and $\sigma >0$ (for zeros of polylogarithms ${\rm{Li}}_s (z)$, we can refer to \cite[Section 8]{OS}).

Let $L(s,\chi) := \sum_{n=1}^\infty \chi (n) n^{-s}$ be the Dirichlet {\textit{L}}-function with a Dirichlet character $\chi$. About 80 years ago, Siegel \cite{Sie} (see also \cite[Theorem 11.11]{MoVau}) showed for any $\varepsilon >0$, there exist $C_\varepsilon>0$ such that, if $\chi$ is a real primitive Dirichlet character modulo $q$, then $L(1,\chi) > C_\varepsilon q^{-\varepsilon}$. It is expected that $L(\sigma,\chi) \ne 0$ for all $0 <\sigma <1$. Namely, it is conjectured that so-called Siegel zeros of Dirichlet {\textit{L}}-functions do not exist. Let $\varphi$ be the Euler totient function and $\chi$ be a primitive Dirichlet character of conductor of $q$. Then the following relations between Dirichlet $L$-functions and Hurwitz zeta functions are well-known. 
\begin{equation*}
\begin{split}
&L (s,\chi) = \sum_{r=1}^q \sum_{n=0}^{\infty} \frac{\chi (r+nq)}{(r+nq)^s} = \sum_{r=1}^q \chi (r) \sum_{n=0}^{\infty} \frac{1}{(r+nq)^s} = q^{-s} \sum_{r=1}^q \chi (r) \zeta (s,r/q) ,\\
&\zeta (s,r/q) = \sum_{n=0}^\infty \frac{1}{(n+r/q)^s} = \sum_{n=0}^\infty \frac{q^s}{(r+qn)^s} =
\frac{q^s}{\varphi (q)} \sum_{\chi \!\!\! \mod q} \overline{\chi} (r) L(s,\chi) .
\end{split}
\end{equation*}
The relations between Hurwitz zeta functions and polylogarithms are expressed as
$$
\zeta (s,r/q) = \sum_{n=1}^q e^{-2\pi i rn/q} {\rm{Li}}_s  (e^{2\pi i rn/q}), \qquad
{\rm{Li}}_s  (e^{2\pi i r/q}) = q^{-s} \sum_{n=1}^q e^{2\pi irn/q} \zeta (s,n/q). 
$$
Hence, we have the following relations between Dirichlet $L$-functions and polylogarithms
\begin{equation*}
\begin{split}
&L (s,\chi) = \frac{1}{G(\overline{\chi})} \sum_{n=1}^{\infty} \sum_{r=1}^q  \frac{\overline{\chi} (r) e^{2\pi i rn/q}}{n^s} 
= \frac{1}{G(\overline{\chi})} \sum_{r=1}^q \overline{\chi} (r) {\rm{Li}}_s  (e^{2\pi i r/q}) ,\\
&{\rm{Li}}_s  (e^{2\pi i r/q}) = q^{-s} \sum_{n=1}^q e^{2\pi irn/q} \zeta (s,n/q) =
\frac{1}{\varphi (q)} \sum_{\chi \!\!\! \mod q} G(\overline{\chi}) L(s,\chi) ,
\end{split}
\end{equation*}
where $G(\overline{\chi}) := \sum_{n=1}^q \overline{\chi}(n)e^{2\pi irn/q}$ denotes the Gauss sum associated to $\overline{\chi}$. It should be emphasized that we have $\zeta (\sigma,a) \ne 0$ for all $0 <\sigma <1$ if and only if $a \ge 1/2$ and ${\rm{Li}}_\sigma (z) \ne 0$ for all $0 <\sigma <1$ and $|z| \le 1$ from Theorem \ref{th:hlz1} despite of the six relations above and the difficulty of the Siegel zero's problem (see also the paper \cite{CGPS} by Conrey, Granville, Poonen and Soundararajan).

%%%%%%%%%%%%%%%%%%%%%%%%%%%%%%%%%%%%%%%%%%%%%%%%%%%%%%%%%%%%%%%%%%%%%%%%%%%%%%%%%%%%
\subsection{Reals zeros of Hurwitz-Lerch type of Euler-Zagier double zeta functions}
%%%%%%%%%%%%%%%%%%%%%%%%%%%%%%%%%%%%%%%%%%%%%%%%%%%%%%%%%%%%%%%%%%%%%%%%%%%%%%%%%%%%
As a double sum and two variable version of the Hurwitz-Lerch zeta function $\Phi(s,a,z)$, we define the following function. 
\begin{definition}[see {\cite[(1)]{KomoQua}}]\label{def:ler}
For $0 < a \le 1$, $s_1,s_2, z_1, z_2 \in {\mathbb{C}}$ and $0<|z_1|, |z_2| \le 1$, the Hurwitz-Lerch type of Euler-Zagier double zeta function $\Phi_2(s_1,s_2,a,z_1,z_2)$ is defined by
\begin{equation}\label{eq:defezhdz}
\Phi_2(s_1,s_2,a,z_1,z_2) := \sum_{m=0}^\infty \frac{z_1^m}{(m+a)^{s_1}} \sum_{n=1}^\infty \frac{z_2^{n-1}}{(m+n+a)^{s_2}}.
\end{equation}
\end{definition}
The function $\Phi_2 (s_1,s_2,a,z_1,z_2)$ can be continued meromorphically to the whole space ${\mathbb{C}}^2$ by Komori's result \cite[Theorem 3.14]{KomoQua} (see also Lemma \ref{lem:ezhd1} and Proposition \ref{pro:ezh2}). In this paper, we prove the following theorem. 
\begin{theorem}\label{th:hldz1}
We have the following:\\
${\rm{(1)}}$. Let $z_1=z_2=1$. Then $\Phi_2 (\sigma_1,\sigma_2,a,1,1) \ne 0$ for all $0 <\sigma_1 <1$, $\sigma_2>1$ and $1 < \sigma_1+\sigma_2 <2$  if and only if $a \ge 1/2$. \\
${\rm{(2)}}$. Let $z_1=1$ and $z_2 \ne 1$. Then $\Phi_2 (\sigma_1,\sigma_2,a,1,z_2) \ne 0$ for all $\sigma_1 > 1$, $\sigma_2 >0$ and $0 < a \le 1$. \\
${\rm{(3)}}$. Let $z_1 \ne 1$ and $z_2 = 1$. Then $\Phi_2 (\sigma_1,\sigma_2,a,z_1,1) \ne 0$ for all $\sigma_1 >0$, $\sigma_2 >1$ and $0 < a \le 1$. \\
${\rm{(4)}}$. Let $z_1 \ne 1$ and $z_2 \ne 1$. Then $\Phi_2 (\sigma_1,\sigma_2,a,z_1,z_2) \ne 0$ for all $\sigma_1 >0$, $\sigma_2 >0$ and $0 < a \le 1$. 
\end{theorem}

When $z_1=z_2=1$, the function above can be regarded as a special case of \cite[(3.2)]{Ma98} which gives an application to special values of Hecke $L$-series of real quadratic fields. Note that Atkinson \cite{Atkinson} obtained an analytic continuation for $\zeta_2(s_1,s_2 \,;a) := \Phi_2(s_1,s_2,a,z_1,z_2)$ with $a=1$ in order to study the mean square $\int_0^T |\zeta (1/2+it)|^2dt$ more than 60 years ago. Matsumoto \cite{Maf04} gave not only an analytic continuation to whole ${\mathbb{C}}^2$ plane but also a functional equation for this kind of zeta functions. Zeros of the Hurwitz type of Euler-Zagier double zeta function $\zeta_2(s_1,s_2 \,;a)$ (after the continuation) at negative integer points are discussed by Akiyama, Egami and Tanigawa \cite{AET}, Akiyama and Tanigawa \cite{Aki2}, Kelliher and Masri \cite{KeMa} and Zhao \cite{Zhao}. We have to remark that Theorem \ref{th:hldz1} (1) indicates the existence of a real zero of $\zeta_2(s_1,s_2 \,;a)$ off negative integer points when $0<a<1/2$. Related to this problem, we have the following proposition. 
\begin{proposition}\label{pro:rezeroezh}
The function $\zeta_2 (\sigma,\sigma \,; a)$ has at least one real zero for $1/2< \sigma <1$. Hence, there exist $1/2 < \sigma_1,\sigma_2 <1$ such that $\zeta_2 (\sigma_1,\sigma_2 \,; a)=0$. 
\end{proposition}

%%%%%%%%%%%%%%%%
\section{Proofs}
%%%%%%%%%%%%%%%%
%%%%%%%%%%%%%%%%%%%%%%%%%%%%%%%%%%%%%%%%%%%
\subsection{Proof of Theorem \ref{th:hlz1}}
%%%%%%%%%%%%%%%%%%%%%%%%%%%%%%%%%%%%%%%%%%%
In order to prove (1) of Theorem \ref{th:hlz1}, we define $H(a,x)$ by
\begin{equation}
\label{eq:defHax}
H(a,x) := \frac{e^{(1-a)x}}{e^x-1} - \frac{1}{x} =  \frac{xe^{(1-a)x} - e^x +1}{x(e^x-1)}, \qquad x>0.
\end{equation}

\begin{lemma}
\label{lem:12.2ac}
For $0 < \sigma <1$ we have the integral representation
\begin{equation}
\label{eq:gamhurac}
\Gamma (s) \zeta (s,a) =  \int_0^\infty \biggl( \frac{e^{(1-a)x}}{e^x-1} - \frac{1}{x} \biggr) x^{s-1} dx = 
\int_0^\infty \!\!\! H(a,x) x^{s-1} dx.
\end{equation}
\end{lemma}
\begin{proof}
When $\sigma >1$, it is well-known that 
$$
\Gamma (s) \zeta (s,a) = \int_0^\infty \frac{x^{s-1}e^{(1-a)x}}{e^x-1} dx 
$$
(see \cite[Theorem 12.2]{Apo}). Hence we have
\begin{equation}
\label{eq:key1}
\begin{split}
\Gamma (s) \zeta (s,a) = & 
\int_0^1 \frac{x^{s-1}e^{(1-a)x}}{e^x-1} dx + \int_1^\infty \frac{x^{s-1}e^{(1-a)x}}{e^x-1} dx \\ = &
\int_0^1 \biggl( \frac{e^{(1-a)x}}{e^x-1} - \frac{1}{x} \biggr) x^{s-1} dx + 
\int_0^1 x^{s-2} dx + \int_1^\infty \frac{x^{s-1}e^{(1-a)x}}{e^x-1} dx \\ = &
\int_0^1 \biggl( \frac{e^{(1-a)x}}{e^x-1} - \frac{1}{x} \biggr) x^{s-1} dx +
\int_1^\infty \frac{x^{s-1}e^{(1-a)x}}{e^x-1} dx + \frac{1}{s-1} .
\end{split}
\end{equation}
By the Taylor expansion of $e^x$, we have
\begin{equation}
\label{eq:haxtay}
\begin{split}
H(a,x) = & \,
\frac{x(\sum_{n=0}^\infty (1-a)^n x^n/n!) - \sum_{n=1}^\infty x^n/n!}{x(\sum_{n=1}^\infty x^n/n!)} \\= & \,
\frac{(1/2-a)x^2 + ((1-a)^2/2!-1/3!)x^3+ \cdots}{x^2+ x^3/2! + \cdots}. 
\end{split}
\end{equation}
Hence, for $\sigma >0$, it holds that
\begin{equation}\label{ieq:abcon1}
\int_0^1 \biggl| \frac{e^{(1-a)x}}{e^x-1} - \frac{1}{x} \biggr| \bigl| x^{s-1} \bigr| dx \ll
\int_0^1 x^{\sigma -1} dx < \infty .
\end{equation}
On the other hand, we have
$$
\frac{1}{s-1} = - \int_1^\infty \frac{x^{s-1}}{x} dx, \qquad 0 < \sigma <1. 
$$
Moreover, for $0 < \sigma <1$ one has
\begin{equation}\label{ieq:abcon2}
\begin{split}
&\int_1^\infty \biggl| \frac{e^{(1-a)x}}{e^x-1} \biggr| \bigl| x^{s-1} \bigr| dx \ll
\int_1^\infty \frac{e^{(1-a)x}}{e^x-1}x^{\sigma -1} dx < \infty , \\
&\int_1^\infty \bigl| x^{s-2} \bigr| dx = \int_1^\infty x^{\sigma-2}  dx = \frac{1}{1-\sigma} < \infty. 
\end{split}
\end{equation}
Therefore, the integral representation 
$$
\Gamma (s) \zeta (s,a) = \int_0^1 \biggl( \frac{e^{(1-a)x}}{e^x-1} - \frac{1}{x} \biggr) x^{s-1} dx +
\int_1^\infty \frac{x^{s-1}e^{(1-a)x}}{e^x-1} dx - \int_1^\infty \frac{x^{s-1}}{x} dx
$$
gives an analytic continuation for $0<\sigma <1$. Thus we obtain this Lemma. 
\end{proof}

\begin{lemma}\label{lem:negdefi}
The function $H(a,x)$ defined by (\ref{eq:defHax}) is negative for all $x>0$ if and only if $a \ge 1/2$. 
\end{lemma}
\begin{proof}
First suppose $0<a<1/2$. Then we have $\lim_{x \to +0} H(a,x)= 1/2 -a >0$ by (\ref{eq:haxtay}). Besides, one has $$
h(a,x):= x(e^x-1)H(a,x) = xe^{(1-a)x} - e^x +1<0
$$
when $x$ is sufficiently large by $e^{1-a} < e$. Hence the function $H(a,x)$ is not negative definite when $0<a<1/2$. 

Next suppose $a\ge 1/2$. Obviously, we have $x(e^x-1) >0$ for all $x>0$. Thus we only have to consider $h(a,x)$ which is the numerator of $H(a,x)$. It holds that $h(a,0)=0$. Thus we show the inequality
$$
h'(a,x) = (1-a)xe^{(1-a)x} + e^{(1-a)x} -e^x < 0 , \qquad x>0 .
$$
This inequality is equivalent to $(1-a)xe^{(1-a)x} + e^{(1-a)x} < e^x$, namely, $1+ (1-a)x< e^{ax}$. We can prove this inequality by the assumption $1-a \le a$ and the Taylor expansion of $e^{ax}= \sum_{n=0}^\infty (a x)^n/n!$. 
\end{proof}

\begin{proof}[Proof of (1) of Theorem \ref{th:hlz1}]
Let $0<a<1/2$. Then we have
$$
\zeta (0,a) = \frac{1}{2} -a > 0
$$
(see \cite[p.~268]{Apo}). Moreover, for any integer $N \ge 0$ and $\sigma >0$, we have
\begin{equation}\label{eq:268}
\zeta (s,a) = 
\sum_{n=0}^N \frac{1}{(n+a)^s} + \frac{(N+a)^{1-s}}{s-1} - s \int_N^\infty \frac{x - [x]}{(x+a)^{s+1}} dx, 
\end{equation}
where $[x]$ denotes the maximal integer less than or equal to $x$ (see \cite[Theorem 12.21]{Apo}). Thus it holds that $\zeta (\sigma,a) \in {\mathbb{R}}$ when $\sigma \in (0,1)$ and 
$$
\lim_{\sigma \to 1-0} \zeta (\sigma,a) = -\infty. 
$$
Hence $\zeta (s,a)$ has at least one zero in the interval $(0,1)$ when $0<a<1/2$. 

Secondly suppose $a \ge 1/2$. Then we have 
$$
\Gamma (\sigma) \zeta (\sigma,a) = 
\int_0^\infty \biggl( \frac{e^{(1-a)x}}{e^x-1} - \frac{1}{x} \biggr) x^{\sigma-1} dx,
\qquad 0 < \sigma < 1
$$
by the integral representation (\ref{eq:gamhurac}). It is well-known that $\Gamma (\sigma) >0$ for any $0 < \sigma < 1$. Thus we obtain 
\begin{equation}\label{ieq:hzn}
\zeta (\sigma,a) <0
\end{equation}
for all $0 < \sigma < 1$ by Lemma \ref{lem:negdefi} and the integral representation above. Therefore $\zeta (\sigma,a)$ does not vanish in the interval $(0,1)$ when $a \ge 1/2$. 
\end{proof}

Next we quote the following integral representation of Hurwitz-Lerch zeta function $\Phi (s,a,z)$ to show (2) of Theorem \ref{th:hlz1}.
\begin{lemma}[{see \cite[p.~53, (3)]{Er}}]
When $z \ne 1$, we have
\begin{equation}\label{eq:hlzinrp1}
\Phi (s,a,z) = \frac{1}{\Gamma (s)} \int_0^\infty \frac{x^{s-1} e^{(1-a)x}}{e^x-z} dx, \qquad \Re (s) >0. 
\end{equation}
\end{lemma}
\begin{proof}[Proof of (2) of Theorem \ref{th:hlz1}]
It should be noted that the integral representation (\ref{eq:hlzinrp1}) converges absolutely for $\sigma >0$ when $z \ne 1$ since one has $e^x-z \ne 0$ for any $x \ge 0$ and
\begin{equation}\label{in:pfabcon}
|\Phi (s,a,z) \Gamma (s)| \le 
\int_0^1 \frac{x^{\sigma-1} e^{(1-a)x}}{|e^x-z|} dx + \int_1^\infty \frac{x^{\sigma-1} e^{(1-a)x}}{|e^x-z|} dx
< \infty . 
\end{equation}
First suppose $z \in [-1,1)$. Then we have $e^{(1-a)x}>0$ and $e^x-z>0$ for all $x \ge 0$. Hence, for any $\sigma >0$, $0 < a \le 1$ and $z \in [-1,1)$,  we have
\begin{equation}\label{ineq:hlzposi1}
\Phi (\sigma,a,z)  >0. 
\end{equation}
Next suppose $z$ is not real. Then it holds that
$$
\Phi (s,a,z) = \frac{1}{\Gamma (s)} \int_0^\infty \frac{x^{s-1} e^{(1-a)x} (e^x-\overline{z})}{|e^x-z|^2} dx, 
\qquad \Re (s) >0,
$$
where $\overline{z}$ is the complex conjugate of $z$, from (\ref{eq:hlzinrp1}). Obviously we have
\begin{equation}\label{ineq:im1}
\Im \bigl( e^{(1-a)x} (e^x-\overline{z}) \bigr) = - e^{(1-a)x} \Im (\overline{z}) \,\,
\begin{matrix}
>0 & \mbox{if } \Im (\overline{z}) <0 \\
<0 & \mbox{if } \Im (\overline{z}) >0
\end{matrix}
\end{equation}
for all $x>0$. Therefore, it holds that
$$
\Im \bigl(\Phi (\sigma,a,z)\bigr) \ne 0
$$ 
for any $\sigma >0$ and $0 < a \le 1$ when $z$ is not real. 
\end{proof}

%%%%%%%%%%%%%%%%%%%%%%%%%%%%%%%%%%%%%%%%%%%%%%%%%%%%%%%%%%%%%%%%%%%%%%%%%%%%%%%%%
\subsection{Proofs of Theorem \ref{th:hldz1} and Proposition \ref{pro:rezeroezh}}
%%%%%%%%%%%%%%%%%%%%%%%%%%%%%%%%%%%%%%%%%%%%%%%%%%%%%%%%%%%%%%%%%%%%%%%%%%%%%%%%%
First, we show that the series expression (\ref{eq:defezhdz}) with $z_1=z_2=1$ converges absolutely when $\Re (s_1) >0$, $\Re (s_2) >1$ and $\Re (s_1+s_2)>2$. 
\begin{lemma}\label{lem:absezh}
For $\Re (s_1) >0$, $\Re (s_2) >1$ and $\Re (s_1+s_2)>2$, the series 
\begin{equation}\label{ser:1}
\sum_{m=0}^\infty \frac{1}{(m+a)^{s_1}} \sum_{n=1}^\infty \frac{1}{(m+n+a)^{s_2}} = 
\sum_{0\le n_1 < n_2} \frac{1}{(n_1+a)^{s_1}(n_2+a)^{s_2}}
\end{equation}
converges absolutely. 
\end{lemma}
\begin{proof}
For any $\varepsilon >0$, it holds that
$$
\sum_{n_2 > n_1 \ge 0} \frac{1}{(n_2+a)^{\sigma_2}(n_1+a)^{\sigma_1}} =
\sum_{n_2 =1}^\infty \frac{\sum_{n_1=0}^{n_2-1}(n_1+a)^{-\sigma_1}}{(n_2+a)^{\sigma_2}} \ll
\sum_{n_2 =1}^\infty \frac{\max \{1,n_2^{1-\sigma_1+\varepsilon}\}}{(n_2+a)^{\sigma_2}}.
$$
Hence the series above converges absolutely in the region $\Re (s_1) >0$, $\Re (s_2) >1$ and $\Re (s_1+s_2)>2$.
\end{proof}

We prove the following integral representation of $\Gamma (s_1) \Gamma (s_2) \zeta (s_1,s_2 \,; a)$ with $\Re (s_1) >0$, $\Re (s_2) >1$ and $\Re (s_1+s_2)>2$ which is a special case of \cite[(3.4)]{Ma98}.
\begin{lemma}
\label{lem:ezh1}
For $\Re (s_1) >0$, $\Re (s_2) >1$ and $\Re (s_1+s_2)>2$, we have the integral representation
\begin{equation}
\label{eq:gamezhdz}
\Gamma (s_1) \Gamma (s_2) \zeta_2 (s_1,s_2 \,; a) = 
\int_0^\infty \frac{y^{s_2-1}}{e^y-1} \int_0^\infty \frac{x^{s_1-1}e^{(1-a)(x+y)}}{e^{x+y}-1} dx dy.
\end{equation}
\end{lemma}
\begin{proof}
For reader's convenience, we prove Lemma \ref{lem:ezh1} although we obtain a proof of this lemma by the proof in \cite[p.~391]{Ma98}. When $y>0$ and $\sigma_1 >0$, one has
$$
\int_0^\infty \frac{x^{s_1-1}e^{(1-a)(x+y)}}{e^{x+y}-1} dx = 
\sum_{m=0}^\infty \int_0^\infty e^{-(a+m)(x+y)} x^{s_1-1} dx = 
\Gamma (s_1) \sum_{m=0}^\infty e^{-(a+m)y} (a+m)^{-s_1} .
$$
Hence, if $\Re (s_1) >0$, $\Re (s_2) >1$ and $\Re (s_1+s_2)>2$, we have
$$
\int_0^\infty \frac{y^{s_2-1}}{e^y-1} \int_0^\infty \frac{x^{s_1-1}e^{(1-a)(x+y)}}{e^{x+y}-1} dx dy =
\Gamma (s_1) \sum_{m=0}^\infty (a+m)^{-s_1} \int_0^\infty \frac{y^{s_2-1}e^{-(a+m)y}}{e^y-1}dy
$$
where the change of the integration and the summation is justified by Lebesgue's dominated convergence theorem and\begin{equation*}
\begin{split}
&\sum_{m=0}^\infty (a+m)^{-\sigma_1} \int_0^\infty \frac{y^{\sigma_2-1}e^{-(a+m)y}}{|e^y-1|}dy\\
&\ll \sum_{m=0}^\infty (a+m)^{-\sigma_1} \int_0^1 e^{-(a+m+1)y} y^{\sigma_2-2} dy
+ \sum_{m=0}^\infty (a+m)^{-\sigma_1} \int_1^\infty e^{-(a+m+1)y} y^{\sigma_2-1} dy \\
&\ll \Gamma (\sigma_2-1) \sum_{m=0}^\infty (a+m)^{-\sigma_1-\sigma_2+1} + 
\Gamma (\sigma_2) \sum_{m=0}^\infty (a+m)^{-\sigma_1-\sigma_2} .
\end{split}
\end{equation*}
Moreover, we have
$$
\int_0^\infty \frac{y^{s_2-1}e^{-(a+m)y}}{e^y-1}dy = 
\sum_{n=1}^\infty \int_0^\infty e^{-(a+m+n)y} y^{s_2-1} dy = 
\Gamma (s_2) \sum_{n=1}^\infty (a+m+n)^{-s_2} 
$$
when $\Re (s_2) >1$. Therefore we obtain this lemma.
\end{proof}

When $0< \Re (s_1) <1$, $\Re (s_2) >1$ and $1<\Re (s_1+s_2)<2$, we have the following integral representation of $\Gamma (s_1) \Gamma (s_2) \zeta_2 (s_1,s_2 \,; a)$ which is a key for the proof of (1) of Theorem \ref{th:hldz1}. It should be mentioned that the series (\ref{ser:1}) with $z_1=z_2=1$ does not converge absolutely in this case since we have
$$
\sum_{n_2 > n_1 \ge 0} \frac{1}{(n_2+a)^{\sigma_2}(n_1+a)^{\sigma_1}} =
\sum_{n_2 =1}^\infty \frac{\sum_{n_1=0}^{n_2-1}(n_1+a)^{-\sigma_1}}{(n_2+a)^{\sigma_2}} \gg
\sum_{n_2 =1}^\infty \frac{n_2^{1-\sigma_1}}{(n_2+a)^{\sigma_2}} = \infty .
$$
 \begin{proposition}
\label{pro:ezh2}
For $0< \Re (s_1) <1$, $\Re (s_2) >1$ and $1<\Re (s_1+s_2)<2$, we have the integral representation
\begin{equation}
\label{eq:gamhur}
\begin{split}
&\Gamma (s_1) \Gamma (s_2) \zeta_2 (s_1,s_2 \,; a) = \\
&\int_0^\infty \frac{y^{s_2-1}}{e^y-1} \int_0^\infty \!\!\! H(a,x+y)x^{s_1-1} dx dy +
\Gamma (s_1) \Gamma (1-s_1) \int_0^\infty \!\!\! H(1,y)y^{s_1+s_2-2} dy,
\end{split}
\end{equation}
where $H(a,x)$ is defined by (\ref{eq:defHax}). Moreover, one has the integral representation
\begin{equation}
\label{eq:gamhur*}
\begin{split}
&\Gamma (s_1) \Gamma (s_2) \zeta_2 (s_1,s_2 \,; a) = \\
&\int_0^\infty \!\!\! \int_0^\infty \frac{y^{s_2-1}}{e^y-1} H(a,x+y)x^{s_1-1} dx dy +
\int_0^\infty \!\!\! \int_0^\infty \frac{x^{s_1-1}}{x+y} H(1,y)y^{s_2-1} dx dy
\end{split}
\end{equation}
for $0< \Re (s_1) <1$, $\Re (s_2) >1$ and $1<\Re (s_1+s_2)<2$.
\end{proposition}

\begin{proof}
It is known (see for example \cite[p.~392]{Ma98}) that
\begin{equation}\label{eq:gain}
\int_0^\infty \frac{x^{s_1-1}}{x+y} dx = y^{s_1-1} \Gamma (s_1) \Gamma (1-s_1), 
\qquad y>0, \quad 0 < \Re (s_1) <1.
\end{equation}
Hence, for $0< \Re (s_1) < 1$, $\Re (s_2) >1$ and $\Re (s_1+s_2)>2$, we have
\begin{equation*}
\begin{split}
&\Gamma (s_1) \Gamma (s_2) \zeta (s_1,s_2 \,; a) \\
&=\int_0^\infty \frac{y^{s_2-1}}{e^y-1} \int_0^\infty
\biggl( \frac{e^{(1-a)(x+y)}}{e^{x+y}-1} - \frac{1}{x+y} \biggr) x^{s_1-1} dx dy +
\int_0^\infty \frac{y^{s_2-1}}{e^y-1} \int_0^\infty \frac{x^{s_1-1}}{x+y} dx dy \\
&=\int_0^\infty \frac{y^{s_2-1}}{e^y-1} \int_0^\infty \!\!\! H(a,x+y)x^{s_1-1} dx dy +
\Gamma (s_1) \Gamma (1-s_1) \int_0^\infty \frac{y^{s_1+s_2-2}}{e^y-1} dy
\end{split}
\end{equation*}
by Lemma \ref{lem:ezh1}. It should be noted that this formula is a special case of \cite[(3.8)]{Ma98} in the region $0< \Re (s_1) < 1$, $\Re (s_2) >1$ and $\Re (s_1+s_2)>2$. Now we show the first integral in the formula above converges absolutely when $0< \Re (s_1) < 1$ and $\Re (s_2) >1$. This is proved as follows. Divide it into two integrals $\int_0^1 \int_0^1$ and $\iint_{{\mathbb{R}}_+^2 \setminus D_1}$, where ${\mathbb{R}}_+$ is the set of all positive real numbers and $D_1 := \{ x,y \in {\mathbb{R}} : 0 <x,y \le 1\}$. Then we have
$$
\int_0^1 \!\!\! \int_0^1 \frac{y^{\sigma_2-1}}{e^y-1} |H(a,x+y)|x^{\sigma_1-1} dx dy < \infty
$$
by using (\ref{eq:haxtay}). Obviously, we have $|H(a,x+y)|\ll (x+y)^{-1} < x^{-1}$ when $x\ge 1$ and $y>0$, and $|H(a,x+y)|\ll (x+y)^{-1} < y^{-1}$ when $0<x<1$ and $y\ge 1$. Hence one has
\begin{equation*}
\begin{split}
&\iint_{{\mathbb{R}}_+^2 \setminus D_1} \frac{y^{\sigma_2-1}}{e^y-1} |H(a,x+y)|x^{\sigma_1-1} dx dy \ll 
\int_1^\infty \!\!\! \int_1^\infty \frac{y^{\sigma_2-1}}{e^y-1} x^{\sigma_1-2} dx dy \\
&+ \int_1^\infty \frac{y^{\sigma_2-2}}{e^y-1} dy \int_0^1 x^{\sigma_1-1} dx dy
+\int_0^1 \frac{y^{\sigma_2-2}}{e^y-1} dy \int_1^\infty x^{\sigma_1-2} dx 
<\infty.
\end{split}
\end{equation*}
Next consider the second integral of (\ref{eq:gamhur}). From the view of (\ref{eq:key1}), one has
$$
\int_0^\infty \frac{y^{s_1+s_2-2}}{e^y-1} dy =
\int_0^1 \! H(1,y)y^{s_1+s_2-2} dy + \int_1^\infty \frac{y^{s_1+s_2-2}}{e^y-1} dy + \frac{1}{s_1+s_2-2}.
$$
The two integrals in the formula above converges absolutely when $1 < \Re (s_1+s_2) <2$ by (\ref{ieq:abcon1}) and (\ref{ieq:abcon2}). On the other hand, we have
$$
\frac{1}{s_1+s_2-2} = - \int_1^\infty \frac{y^{s_1+s_2-2}}{y} dy, \qquad 1 < \Re (s_1+s_2) <2. 
$$
Obviously, the integral $\int_1^\infty y^{s_1+s_2-3} dy$ converges absolutely when $1 < \Re (s_1+s_2) <2$. Therefore, we obtain (\ref{eq:gamhur}) by the definition of $H(1,y)$.

It was shown in the proof of (\ref{eq:gamhur}) that the first double integrals converges absolutely when $0< \Re (s_1) < 1$ and $\Re (s_2) >1$. If we can interchange of the order of the second double integrations, we have (\ref{eq:gamhur*}) from (\ref{eq:gamhur}) and (\ref{eq:gain}). This is justified as follows. By using (\ref{eq:gamhur}), we have
$$
\int_0^\infty \left| \frac{x^{s_1-1}}{x+y} \right| dx = \int_0^\infty \frac{x^{\sigma_1-1}}{x+y} dx = 
y^{\sigma_1-1} \Gamma (\sigma_1) \Gamma (1-\sigma_1) < \infty
$$
for $y>0$ and $0 < \sigma_1 <1$. Furthermore, it holds that
$$
\int_0^\infty \biggl| \frac{H(1,y)}{x+y} y^{s_1+s_2-1} \biggr| dy \le 
\int_0^\infty \!\!\! H(1,y)y^{\sigma_1+\sigma_2-2} dy < \infty
$$ 
when $x \ge 0$ and $1< \Re (s_1+s_2) < 2$ from the view of the proof of (\ref{ieq:abcon1}) and (\ref{ieq:abcon2}). Thus we can apply Fubini's theorem. 
\end{proof}

We quote the following Lemma form Akiyama and Ishikawa \cite{AI}. We have to remark that Akiyama and Ishikawa \cite{AI} consider not $\sum_{0 \le n_1 < n_2}$ but $\sum_{0 < n_1 < n_2}$ in the definition of the double zeta-function. Note that we have
\begin{equation*}
\begin{split}
\zeta_2 (s_1, s_2 \, ; a) &=
\sum_{0 < n_2} \frac{1}{a^{s_1}(n_2+a)^{s_2}} + \sum_{0 < n_1 < n_2} \frac{1}{(n_1+a)^{s_1}(n_2+a)^{s_2}} \\ 
&= a^{-s_1} \bigl( \zeta (s_2,a) - a^{-s_2} \bigr) + \sum_{0 < n_1 < n_2} \frac{1}{(n_1+a)^{s_1}(n_2+a)^{s_2}}.
\end{split}
\end{equation*}

\begin{lemma}[{see \cite[(15)]{AI}}]
Let $\lambda >0$, $l \in {\mathbb{N}}_0$, $\widetilde{B}_l(x):=B_l(x-[x])$, where $B_l(x)$ be the $l$-th Bernoulli polynomial and
$$
\Phi_l(s \,|\, \lambda,a) := \frac{(s)_{l+1}}{(l+1)!} 
\int_\lambda^\infty \frac{\widetilde{B}_{l+1}(x)}{(x+a)^{s+l+1}} dx, \qquad
(s)_l :=
\begin{cases}
s(s+1)\cdots (s+l-1) & l \ge 1,\\
1 & l=0 .
\end{cases}
$$
Then, for $\Re (s_1+s_2)>-l$, we have
\begin{equation}
\label{eq:AI1}
\begin{split}
\zeta_2 (s_1,s_2 \,; a) &= a^{-s_1} \bigl( \zeta (s_2,a) - a^{-s_2} \bigr) + 
\frac{\zeta(s_1+s_2-1,a) - a^{1-s_1-s_2}}{s_2-1} \\ &+ 
\sum_{r=0}^l \frac{B_{r+1}(0)}{(r+1)!} (s_2)_r \bigl( \zeta(s_1+s_2+r,a) - a^{-s_1-s_2-r} \bigr)-
\sum_{n=1}^\infty \frac{\Phi_l (s_2 \,|\, n,a)}{(n+a)^{s_1}} .
\end{split}
\end{equation}
The last summation is absolutely convergent, and hence holomorphic, in $\Re (s_1+s_2)>-l$. 
\end{lemma}

\begin{proof}[Proof of (1) of Theorem \ref{th:hldz1}]
By putting $l=0$ in (\ref{eq:AI1}), we have
\begin{equation*}
\begin{split}
\zeta_2 (s_1,s_2 \,; a) =& \,
a^{-s_1} \bigl( \zeta (s_2,a) - a^{-s_2} \bigr) + \frac{\zeta(s_1+s_2-1,a)- a^{1-s_1-s_2}}{s_2-1} \\ & -
\frac{\zeta(s_1+s_2,a) - a^{-s_1-s_2}}{2} -
\sum_{n=1}^\infty \frac{\Phi_0 (s_2 \,|\, n,a)}{(n+a)^{s_1}} .
\end{split}
\end{equation*}
Note that the last sum converges when $\Re (s_1+s_2)>0$. Hence we have
\begin{equation}\label{lim:ep}
\lim_{\varepsilon \to +0} \varepsilon \biggl( \frac{\zeta(s_1+s_2,a) - a^{-s_1-s_2}}{2} +
\sum_{n=1}^\infty \frac{\Phi_0 (s_2 \,|\, n,a)}{(n+a)^{s_1}} \biggr) =0
\end{equation}
for $\Re (s_1+s_2)>1+\delta$, where $\delta >0$. By using (\ref{eq:268}) and (\ref{lim:ep}), we have
$$
\lim_{\varepsilon \to +0} \varepsilon \zeta_2(1-2\varepsilon , 1+ \varepsilon \,; a) = 
\lim_{\varepsilon \to +0} 
\bigl( \varepsilon a^{2\varepsilon-1} \zeta(1+\varepsilon,a) + 
(\zeta(1-\varepsilon,a) - a^{\varepsilon -1}) \bigr) = -\infty
$$
since $1-2\varepsilon + 1+ \varepsilon > 1+\delta$ for some $\delta >0$. Thus one has 
$$
\lim_{\varepsilon \to +0} \zeta_2(1-2\varepsilon , 1+ \varepsilon \,;a) = -\infty.
$$

First suppose $0<a <1/2$. Then we have $\zeta (0,a) = 1/2-a>0$. Hence there exists $0<\sigma_0<1$ such that $\zeta (\sigma_0,a)>0$ when $0<a <1/2$. Then one has
$$
\lim_{\varepsilon \to +0} \varepsilon \zeta_2(\sigma_0-\varepsilon , 1+ \varepsilon \,; a) = 
a^{-\sigma_0} + (\zeta (\sigma_0,a) - a^{-\sigma_0}) = \zeta (\sigma_0,a) > 0
$$
from (\ref{eq:268}) and (\ref{lim:ep}). Hence $\lim_{\varepsilon \to +0} \zeta_2(\sigma_0-\varepsilon , 1+ \varepsilon \,; a) = + \infty $ when $0<a <1/2$. Therefore $\zeta_2 (s_1,s_2\,;a)$ has at least one real zero in $0<\sigma_1<1$, $\sigma_2>1$ and $1<\sigma_1+\sigma_2<2$ when $0 < a <1/2$ by the intermediate value theorem. 

Next suppose $a \ge 1/2$. By Lemma \ref{lem:negdefi}, we obtain $H(a,x+y)<0$ and $H(1,y)<0$ for all $x,y>0$ when $a \ge 1/2$. Hence one has
\begin{equation}
\label{ineq:ezhnega}
\zeta_2 (\sigma_1,\sigma_2\,;a) < 0, \qquad 0<\sigma_1<1, \quad \sigma_2>1, \quad 1<\sigma_1+\sigma_2<2
\end{equation}
when $a \ge 1/2$ from the integral representation (\ref{eq:gamhur*}). Therefore, the function $\zeta_2 (\sigma_1,\sigma_2\,;a)$ with $a \ge 1/2$ does not vanish when $0<\sigma_1<1$, $\sigma_2>1$ and $1<\sigma_1+\sigma_2<2$. 
\end{proof}

We show the following lemma which is a generalization of Lemma \ref{lem:ezh1}. 
\begin{lemma}
\label{lem:ezhd1}
For $\Re (s_1) >1$ and $\Re (s_2) >1$, we have the integral representation
\begin{equation}
\label{eq:gamezhdz2}
\Gamma (s_1) \Gamma (s_2) \Phi_2 (s_1,s_2, a, z_1,z_2) = 
\int_0^\infty \frac{y^{s_2-1}}{e^y-z_2} \int_0^\infty \frac{x^{s_1-1}e^{(1-a)(x+y)}}{e^{x+y}-z_1} dx dy.
\end{equation}
Furthermore, we have the following:\\
${\rm{(1)}}$. Suppose $z_1=z_2=1$. Then the integral representation (\ref{eq:gamezhdz2}) holds for $\Re (s_1) > 0$, $\Re(s_2)>1$ and $\Re (s_1+s_2) >2$. \\
${\rm{(2)}}$. Suppose $z_1=1$ and $z_2 \ne 1$. Then the integral representation (\ref{eq:gamezhdz2}) holds for $\Re (s_1) > 1$ and $\Re(s_2)>0$. \\
${\rm{(3)}}$. Suppose $z_1 \ne 1$ and $z_2 = 1$. Then the integral representation (\ref{eq:gamezhdz2}) holds for $\Re (s_1) > 0$ and $\Re(s_2)>1$. \\
${\rm{(4)}}$. Suppose $z_1 \ne 1$ and $z_2 \ne 1$. Then the integral representation (\ref{eq:gamezhdz2}) holds for $\Re (s_1) > 0$ and $\Re(s_2)>0$. 
\end{lemma}
\begin{proof}
The statement (1) has already proved in Lemma \ref{lem:ezh1}. Thus we only have to show (2), (3) and (4). First assume $\Re (s_1) >1$ and $\Re (s_2) >1$. For $y>0$, we have 
$$
\int_0^\infty \frac{x^{s_1-1}e^{(1-a)(x+y)}}{e^{x+y}-z_1} dx = 
\sum_{m=0}^\infty \int_0^\infty \frac{z_1^m x^{s_1-1}}{e^{(a+m)(x+y)}} dx = 
\Gamma (s_1) \sum_{m=0}^\infty \frac{z_1^m e^{-(a+m)y}}{(a+m)^{s_1}} .
$$
Hence it holds that
$$
\int_0^\infty \frac{y^{s_2-1}}{e^y-z_2} \int_0^\infty \frac{x^{s_1-1}e^{(1-a)(x+y)}}{e^{x+y}-z_1} dx dy =
\Gamma (s_1) \sum_{m=0}^\infty \frac{z_1^m}{(a+m)^{s_1}} \int_0^\infty \frac{y^{s_2-1}e^{-(a+m)y}}{e^y-z_2}dy
$$
where the change of the integration and the summation is justified by the method used in the proof of Lemma \ref{lem:ezh1}. Moreover, one has
$$
\int_0^\infty \frac{y^{s_2-1}e^{-(a+m)y}}{e^y-z_2}dy = 
\sum_{n=1}^\infty \int_0^\infty \frac{z_2^{n-1} y^{s_2-1}}{e^{(a+m+n)y}} dy = 
\Gamma (s_2) \sum_{n=1}^\infty \frac{z_2^{n-1}}{(a+m+n)^{s_2}} 
$$
for $\Re (s_2) >1$. Therefore we obtain (\ref{eq:gamezhdz2}) when $\Re (s_1) >1$ and $\Re (s_2) >1$.

Let $z_1=1$ and $z_2 \ne 1$. Then the integral representation (\ref{eq:gamezhdz2}) converges absolutely when $\Re (s_1) > 1$ and $\Re(s_2)>0$ since $e^y-z_2 \ne 0$ for any $y \ge 0$ (see (\ref{in:pfabcon})). Similarly, the integral (\ref{eq:gamezhdz2}) converges absolutely when $\Re (s_1) > 0$ and $\Re(s_2)>1$ since $e^{x+y}-z_1 \ne 0$ for all $x,y \ge 0$ when $z_1 \ne 1$ and $z_2 = 1$. Furthermore, the integral (\ref{eq:gamezhdz2}) converges absolutely when $\Re (s_1) > 0$ and $\Re(s_2)>0$ since $e^y-z_2 \ne 0$ and $e^{x+y}-z_1 \ne 0$ for all $x,y \ge 0$ if $z_1 \ne 1$ and $z_2 \ne 1$. Thus we have this lemma. 
\end{proof}

\begin{proof}[Proof of (2), (3) and (4) of Theorem \ref{th:hldz1}]
First, we prove (2) of Theorem \ref{th:hldz1}. Namely, suppose $z_1=1$ and $z_2 \ne 1$. From (\ref{eq:gamezhdz2}) and Fubini's theorem, we have
$$
\Gamma (s_1) \Gamma (s_2) \Phi_2 (s_1,s_2, a, 1,z_2) = 
\int_0^\infty \int_0^\infty \frac{y^{s_2-1}}{e^y-z_2} \frac{x^{s_1-1}e^{(1-a)(x+y)}}{e^{x+y}-1} dx dy
$$
for $\Re (s_1) > 1$ and $\Re(s_2)>0$ since the integral above converges absolutely. Let $z_2 \in [-1,1)$. Then we obtain $\Phi_2 (\sigma_1,\sigma_2, a, 1,z_2) >0$ when $\sigma_1 >1$ and $\sigma_2 >0$ by the fact that $e^y-z_2 >0$ and $e^{x+y}-1 \ge 0$ for any $x,y \ge 0$. 

Next assume that $z_2$ is not real. Then we have $\Im (\Phi_2 (\sigma_1,\sigma_2, a, 1,z_2) ) \ne 0$ for any $\sigma_1 >1$ and $\sigma_2 >0$ by the manner used in the proof of (\ref{ineq:im1}). Hence we have (2) of Theorem \ref{th:hldz1}. The case (3) is proved similarly. Moreover, we can easily show the case (4) when at least one of $z_1$ and $z_2$ are real numbers. 

Finally, suppose that both $z_1$ and $z_2$ are not real. Then we have
\begin{equation*}
\begin{split}
&\Gamma (s_1) \Gamma (s_2) \Phi_2 (s_1,s_2, a, z_1,z_2) = \\ &\int_0^\infty \int_0^\infty 
\frac{x^{s_1-1}y^{s_2-1}e^{(1-a)(x+y)} (e^{x+2y} - \overline{z_1} e^y - \overline{z_2} e^{x+y} + \overline{z_1}\overline{z_2})}{|e^y-z_2|^2 |e^{x+y}-z_1|^2} dx dy.
\end{split}
\end{equation*}
Now put $\overline{z_1} := r_1 e^{2\pi i \theta_1}$ and $\overline{z_2} := r_2 e^{2\pi i \theta_2}$, where $0 < r_1, r_2, \theta_1, \theta_2 \le 1$. Then one has
\begin{equation*}
\begin{split}
&\Re (e^{x+2y} - \overline{z_1} e^y - \overline{z_2} e^{x+y} + \overline{z_1}\overline{z_2}) \\ & =
e^{x+2y}- e^y r_1 \cos (2\pi \theta_1) -e^{x+y} r_2 \cos (2\pi \theta_2) \\ & \qquad + 
r_1r_2 \bigl(\cos (2\pi \theta_1) \cos (2\pi \theta_2) - \sin (2\pi \theta_1) \sin (2\pi \theta_2)\bigr) ,
\end{split}
\end{equation*}
\begin{equation*}
\begin{split}
&\Im (e^{x+2y} - \overline{z_1} e^y - \overline{z_2} e^{x+y} + \overline{z_1}\overline{z_2}) \\ &=
-e^y r_1 \sin (2\pi \theta_1) -e^{x+y} r_2 \sin (2\pi \theta_2) + 
r_1r_2 \bigl(\sin (2\pi \theta_1) \cos (2\pi \theta_2) + \sin (2\pi \theta_2) \cos (2\pi \theta_1)\bigr) \\ &=
-r_1 \sin (2\pi \theta_1) \bigl(e^y - r_2\cos (2\pi \theta_2)\bigr) - r_2 \sin (2\pi \theta_2) \bigl(e^{x+y} - r_1\cos (2\pi \theta_1)\bigr). 
\end{split}
\end{equation*}
When $\sin (2\pi \theta_1) \sin (2\pi \theta_2) >0$, we can see that the sign of $\Im (e^{x+2y} - \overline{z_1} e^y - \overline{z_2} e^{x+y} + \overline{z_1}\overline{z_2})$ does not change even if $x$ and $y$ run through from $0$ to $\infty$. Therefore it holds that $\Im (\Phi_2 (\sigma_1,\sigma_2, a, z_1, z_2) ) \ne 0$ for any $\sigma_1 >0$ and $\sigma_2 >0$ if $\sin (2\pi \theta_1) \sin (2\pi \theta_2) >0$. Hence, we only have to show the case $\sin (2\pi \theta_1) \sin (2\pi \theta_2) <0$. In this case, we have 
\begin{equation*}
\begin{split}
&\Re (e^{x+2y} - \overline{z_1} e^y - \overline{z_2} e^{x+y} + \overline{z_1}\overline{z_2}) \\ &>  
e^{x+2y}- e^y r_1 \cos (2\pi \theta_1) -e^{x+y} r_2 \cos (2\pi \theta_2) + r_1r_2  \cos (2\pi \theta_1) \cos (2\pi \theta_2) \\ &=
\bigl(e^y - r_2 \cos (2\pi \theta_2) \bigr) \bigl(e^{x+y} - r_1 \cos (2\pi \theta_1) \bigr) >0. 
\end{split}
\end{equation*}
Thus we have $\Re (\Phi_2 (\sigma_1,\sigma_2, a, z_1, z_2) ) > 0$ for any $\sigma_1 >0$ and $\sigma_2 >0$ if $\sin (2\pi \theta_1) \sin (2\pi \theta_2) <0$. Therefore, we obtain (2), (3) and (4) of Theorem \ref{th:hldz1}. 
\end{proof}

\begin{proof}[Proof of Proposition \ref{pro:rezeroezh}]
When $\Re (s_1),\Re (s_2)>1$, we have
\begin{equation*}
\begin{split}
&\zeta (s_1,a) \zeta (s_2,a) =
\Biggl(\sum_{m>n\ge 0} + \sum_{n>m\ge 0} + \sum_{m=n\ge 0} \Biggr) \frac{1}{(m+a)^{s_1} (n+a)^{s_2}} \\ &=
\zeta_2 (s_1,s_2 \,; a) + \zeta_2 (s_2,s_1 \,; a) + \zeta (s_1+s_2,a)
\end{split}
\end{equation*}
form (\ref{ser:1}) and the view of the harmonic product. Hence one has
$$
2 \zeta_2 (s,s \,; a) = \zeta (s,a)^2 - \zeta (2s,a) 
$$
when $\Re (s) >1$. Note that the equation above gives an analytic continuation of $\zeta_2 (s,s \,; a)$ for $1/2<\Re (s) <1$. Then we have
$$
\lim_{\sigma \to 1-0} \zeta_2 (\sigma,\sigma \,; a) = \infty, \qquad 
\lim_{\sigma \to 1/2+0} \zeta_2 (\sigma,\sigma \,; a) = -\infty .
$$
by (\ref{eq:268}). Hence we have Proposition \ref{pro:rezeroezh} by the intermediate value theorem. 
\end{proof}

%%%%%%%%%%%%%%%%%%%%%%%%
\subsection*{Acknowledgments}
%%%%%%%%%%%%%%%%%%%%%%%%
The author would like to thank the referee for useful comments and suggestions that helped him to improve the original manuscript. 

%%%%%%%%%%%%%%%%%%%%%%%%%%
 
\end{document}